\documentclass{amsart}
\usepackage{amsmath,amsfonts,amsthm}
\usepackage{amssymb,latexsym}
\usepackage{graphics}
\usepackage[colorlinks]{hyperref}

\usepackage[all]{xypic}

\theoremstyle{plain}
\theoremstyle{definition}
\newtheorem{theorem}{Theorem}[section]
\newtheorem{lemma}[theorem]{Lemma}

\newtheorem{corollary}[theorem]{Corollary}

\newtheorem{note}[theorem]{Note}
\newtheorem{jadval}[theorem]{Table}
\newtheorem{nemoodar}[theorem]{Diagram}

\newtheorem{remark}[theorem]{Remark}
\theoremstyle{remark}

\numberwithin{equation}{section}
\newcommand{\SP}{\: \: \: \: \:}

\title{Set--theoretical entropies of generalized shifts}
\author[Z. Nili Ahmadabadi, F. Ayatollah Zadeh Shirazi]{Zahra Nili Ahmadabadi, Fatemah Ayatollah Zadeh Shirazi}
\begin{document}
\begin{abstract}
In the following text for arbitrary $X$ with at least two elements, nonempty set $\Gamma$
and self-map $\varphi:\Gamma\to\Gamma$ we prove the set-theoretical entropy of 
generalized shift $\sigma_\varphi:X^\Gamma\to X^\Gamma$ ($\sigma_\varphi((x_\alpha)_{\alpha\in\Gamma})=(x_{\varphi(\alpha)})_{\alpha\in\Gamma}$
(for $(x_\alpha)_{\alpha\in\Gamma}\in X^\Gamma$))
is either zero or infinity, moreover it is zero if and only if $\varphi$ is quasi-periodic.
\\
We continue our study on contravariant set-theoretical entropy of generalized shift
and motivate the text using counterexamples dealing with algebraic, topological, set-theoretical
and contravariant  set-theoretical positive entropies of generalized shifts.

\end{abstract}
\maketitle
\noindent {\small {\bf 2010 Mathematics Subject Classification:}  54C70 \\
{\bf Keywords:}} Bounded map, Contravariant set-theoretical entropy, Quasi-periodic, Set-theoretical entropy.
\section{Introduction}
\noindent Amongst the most powerful tools in ergodic theory and dynamical
systems we may mention one-sided shift 
$\mathop{\{1,\ldots,k\}^{\mathbb N}\to\{1,\ldots,k\}^{\mathbb N}}\limits_{\:\:\:\:
(x_n)_{n\geq1}\mapsto(x_{n+1})_{n\geq1}}$ and two-sided shift 
\linebreak
$\mathop{\{1,\ldots,k\}^{\mathbb Z}\to\{1,\ldots,k\}^{\mathbb Z}}\limits_{\:\:\:\:
(x_n)_{n\in{\mathbb Z}}\mapsto(x_{n+1})_{n\in{\mathbb Z}}}$ \cite{walters}.
Now suppose $X$ is an arbitrary set with at least two elements,
$\Gamma$ is a nonempty set, and $\varphi:\Gamma\to\Gamma$ is arbitrary,
then $\sigma_\varphi:X^\Gamma\to X^\Gamma$ with $\sigma_\varphi((x_\alpha)_{\alpha\in 
\Gamma})=(x_{\varphi(\alpha)})_{\alpha\in\Gamma}$ (for $(x_\alpha)_{\alpha\in\Gamma}\in 
X^\Gamma$) is a {\it generalized shift}. Generalized shifts have been introduced for the
first time in \cite{AHK}. It's evident that for self-map $\varphi:\Gamma\to\Gamma$ 
and generalized shift $\sigma_\varphi:X^\Gamma\to X^\Gamma$ if $X$ has a group
(resp. vector space, topological) structure, then $\sigma_\varphi:X^\Gamma\to X^\Gamma$ is a group
homomorphism (resp. linear map, continuous (in which $X^\Gamma$ considered under product
topology)), so many dynamical \cite{dev} and non-dynamical \cite{anna} properties of 
generalized shifts have been studied in several texts. In this text our main aim is to
study set-theoretical and contravariant set-theoretical entropy of generalized shifts.
We complete our investigations with a comparative study regarding 
 set-theoretical, contravariant set-theoretical, topological and algebraic entropies of generalized shifts.
\\
\noindent For self-map $g:A\to A$ and $x,y\in A$ let $x\leq_gy$ if and only if there 
exists $n\geq0$ with $g^n(x)=y$, then $(A,\leq_g)$ is a preordered 
(reflexive and transitive) set. Note that for set $A$ by $|A|$ we mean the cardinality of $A$
if it is finite and $\infty$ otherwise.
\\
Although one may obtain the following lemma using \cite{anna}, we establish it
here directly.
\begin{note}[Bounded self-map, Quasi-periodic self-map]
For self-map $g:A\to A$ the following statements are equivalent
(consider preordered set $(A,\leq_g)$):
\begin{itemize}
\item[1.] there exists $N\geq1$ such that for all totally preordered subset $I$ of $A$
	(reflexive, transitive and for all $x,y\in A$ we have $x\leq_gy$ or $y\leq_g x$) we
	have $|I|\leq N$ (i.e., $g:A\to A$ is {\it bounded} \cite{anna}),
\item[2.] $\sup\left\{|\{g^n(x):n\geq0\}|:x\in A\right\}<\infty$,
\item[3.] there exists $n>m\geq1$ with $g^n=g^m$ ($g$ is {\it quasi-periodic}).
\end{itemize}
\end{note}
\begin{proof}
``(1) $\Rightarrow$ (2)''
Suppose there exists $N\geq1$ such that for all totally 
preordered subset $I$ of $A$ we have $|I|\leq N$.
Choose $x\in A$, then $\{g^n(x):n\geq0\}$ is a totally preordered subset of $A$,
thus $|\{g^n(x):n\geq0\}|\leq N$, hence
$\sup\left\{|\{g^n(y):n\geq0\}|:y\in A\right\}$
$\leq N<\infty$.
\\
``(2) $\Rightarrow$ (3)''
Suppose $\sup\left\{|\{g^n(x):n\geq0\}|:x\in A\right\}=N<\infty$, then for all $x\in A$
we have $\{g^n(x):n\geq0\}=\{x,g(x),\ldots,g^{N-1}(x)\}$ and there exists
$n_x\in\{0,\ldots,N-1\}$ with $g^N(x)=g^{n_x}(x)$, thus $g^{N-n_x}(g^N(x))=g^N(x)$
and
\[\forall z\in g^N(A)\:\:\exists i\in\{1,\ldots,N\}\:(g^i(z)=z)\]
so for all $z\in g^N(A)$ we have $g^{N!}(z)=z$, thus for all $x\in A$ we have
$g^{N!+N}(x)=g^N(x)$.
\\
``(3) $\Rightarrow$ (1)''
Suppose there exist $n>m\geq1$ with $g^n=g^m$ and $I$ is a totally preordered subset of $A$,
choose distinct $x_1,\ldots,x_k\in I$ and suppose $x_1\leq_gx_2\leq_g\cdots\leq_gx_k$.
For all $i\in\{1,\ldots,k\}$ there exists $p_i\geq0$ with $x_i=g^{p_i}(x_1)$,
so $\{x_1,\ldots,x_k\}\subseteq\{g^i(x_1):i\geq0\}=\{g^i(x_1):i\in\{0,\ldots,n\}\}$ and $k\leq n+1$.
Hence $|I|\leq n+1$ which completes the proof.
\end{proof}
\noindent {\bf Convention.} In the following text suppose $X$ is an arbitrary set with at least two elements,
$\Gamma$ is a nonempty set, and $\varphi:\Gamma\to\Gamma$ is arbitrary.
\subsection{Background on set-theoretical entropy}
For self-map $g:A\to A$ and $a\in A$, the set $\{g^n(a):n\geq0\}$ is
the {\it orbit} of $a$, we say $a\in A$ is a {\it wandering point} (or {\it non-quasi periodic point})
of $g$, if $\{g^n(a):n\geq0\}$ is infinite, or equivalently $\{g^n(a)\}_{n\geq1}$ is a 
one-to-one sequence. We denote the collection of all wandering points of $g:A\to A$ with
$W(g)$.
\\
For $g:A\to A$ denote the {\it infinite orbit number of} $g$ by  $\mathfrak{o}(g)$
and define it with $\sup(\{0\}\cup\{k\geq1:\exists a_1,\ldots,a_k\in W(g)\:
(\{g^n(a_1)\}_{n\geq1},\ldots,\{g^k(a)\}_{n\geq1}$ are pairwise disjoint sequences$)\})$,
i.e.  $\mathfrak{o}(g)=\sup(\{0\}\cup\{k\geq1:$ there exists $k$ pairwise disjoint infinite orbits$\})$.
So $W(g)\neq\varnothing$ if and only if $\mathfrak{o}(g)\geq1$.
On the other hand for finite subset $D$ of $A$ the following limit exists \cite{set}:
\[{\rm ent}_{\rm set}(g,D)={\displaystyle\lim_{n\to\infty}\dfrac{|D\cup g(D)\cup\cdots\cup
g^{n-1}(D)|}{n}}\:.\]
Now we call
$\sup\{{\rm ent}_{\rm set}(g,D):D$ is a finite subset of $A\}$ the {\it set-theoretical entropy
of} $g$ and denote it with ${\rm ent}_{\rm set}(g)$. Moreover 
${\rm ent}_{\rm set}(g)=\mathfrak{o}(g)$ \cite{set}.
\subsection{Background on contravariant set-theoretical entropy}
Suppose self-map $g:A\to A$ is onto and finite fibre (i.e., for all $a\in A$, $g^{-1}(a)$ is finite),
then for finite subset $D$ of $A$ the following limit exists \cite{anna2}:
\[{\rm ent}_{\rm cset}(g,D)={\displaystyle\lim_{n\to\infty}\dfrac{|D\cup g^{-1}(D)\cup\cdots\cup
g^{-(n-1)}(D)|}{n}}\]
Now let ${\rm ent}_{\rm cset}(g):=\sup\{{\rm ent}_{\rm cset}(g,D):D$ is a finite subset of $A\}$. If
$k:A\to A$ is an arbitrary finite fibre map, then for {\it surjective cover} of $k$, i.e.
 ${\rm sc}(k):=\bigcap\{k^n(A):n\geq1\}$, the map 
$k\restriction_{{\rm sc}(k)}:{\rm sc}(k)\to{\rm sc}(k)$ is an onto finite fibre map and
${\rm ent}_{\rm cset}(k):={\rm ent}_{\rm cset}(k\restriction_{{\rm sc}(k)})$ is the
{\it contravariant set-theoretical entropy} of $k$. Moreover we say $\{x_n\}_{n\geq1}$
is a $k-${\it anti-orbit sequence} (or simply anti-orbit sequence)
if for all $n\geq1$ we have $k(x_{n+1})=x_n$ and define
{\it infinite anti-orbit number of} $k$ as  
$\mathfrak{a}(k)=\sup(\{0\}\cup\{j\geq1:$ there exists $j$ pairwise disjoint infinite anti-orbits$\})$.
Moreover 
${\rm ent}_{\rm cset}(k)=\mathfrak{a}(k)$ \cite{anna2}.
\section{Set-theoretical entropy of $\sigma_\varphi:X^\Gamma\to X^\Gamma$}
\noindent In this section we prove that 
for generalized shift $\sigma_\varphi:X^\Gamma\to X^\Gamma$, 
${\rm ent}_{\rm set}( \sigma_\varphi)\in\{0,\infty\}$ and ${\rm ent}_{\rm set}( \sigma_\varphi)=0$
if and only if $\varphi$ is quasi-periodic.
\begin{lemma}\label{salam10}
If $W(\varphi)\neq\varnothing$, then $W(\sigma_\varphi)\neq\varnothing$.
\end{lemma}
\begin{proof}
Consider distinct points $p,q\in X$ and $\theta\in W(\varphi)$, thus $(\varphi^n(\theta))_{n\geq0}$
is a one-to-one sequence. Let:
\[x_\alpha:=\left\{\begin{array}{lc} p & \alpha\in\{\varphi^{2^n}(\theta):n\geq1\}\:, \\
q & {\rm otherwise\:,}\end{array}\right.\]
then $(x_\alpha)_{\alpha\in\Gamma}\in W(\sigma_\varphi)$, otherwise
there exists $s>t\geq1$ such that $\sigma_{\varphi}^s((x_\alpha)_{\alpha\in\Gamma})=
\sigma_{\varphi}^t((x_\alpha)_{\alpha\in\Gamma})$, thus $x_{\varphi^s(\alpha)}=x_{\varphi^t(\alpha)}$
for all $\alpha\in\Gamma$.  In particular, $x_{\varphi^{s+i}(\theta)}=x_{\varphi^{t+i}(\theta)}$ for all
$i\geq0$. Choose $j\geq1$ with $j+s\in\{2^n:n\geq1\}$.  We have the following cases:
\\
{\bf Case 1:} $j+t\notin\{2^n:n\geq1\}$.  In this case we have $p=x_{\varphi^{s+j}(\theta)}=
x_{\varphi^{t+j}(\theta)}=q$, which is a contradiction.
\\
{\bf Case 2:} $j+t\in\{2^n:n\geq1\}$.  In this case using $j+t>j+s\in\{2^n:n\geq1\}$ we have
$j+t\geq3$. There exist $k\geq1$ and $l\geq2$ with $j+s=2^k$ and $j+t=2^l$.
Let $i=2j+s$, then $i+s=2(j+s)=2^{k+1}\in\{2^n:n\geq1\}$ and $i+t=2^l+2^k=2^l(1+2^{k-l})\notin 
\{2^n:n\geq1\}$ (note that $k>l$ and $1+2^{k-l}$ is odd). So $p=x_{\varphi^{s+i}(\theta)}=
x_{\varphi^{t+i}(\theta)}=q$, which is a contradiction.
\\
Using the above two cases, we have $(x_\alpha)_{\alpha\in\Gamma}\in W(\sigma_\varphi)$.
\end{proof}
\begin{lemma}\label{salam20}
If $W(\varphi)\neq\varnothing$, then $\mathfrak{o}(\sigma_\varphi)=\infty$.
\end{lemma}
\begin{proof}
Consider $\theta\in\Gamma$ with infinite $\{\varphi^n(\theta):n\geq0\}$ and choose 
distinct $p,q\in X$, thus $(\varphi^n(\theta))_{n\geq0}$ is a one-to-one sequence. For $s\geq1$ let:
\[x_\alpha^s:=\left\{\begin{array}{lc} p & \alpha\in
	\{\varphi^n(\theta):\exists k\geq0\:(ks+\frac{k(k+1)}{2}<i\leq ks+\frac{k(k+1)}{2}+s)\}\:, \\
	q & {\rm otherwise}\:,\end{array}\right.\]
so:

$(x_{\varphi(\theta)}^s,x_{\varphi^2(\theta)}^s,x_{\varphi^3(\theta)}^s,\cdots)$
\[=(\:\underbrace{p,\cdots,p}_{s\:{\rm times}},q,
\underbrace{p,\cdots,p}_{s\:{\rm times}},q,q,\underbrace{p,\cdots,p}_{s\:{\rm times}},q,q,q,
\underbrace{p,\cdots,p}_{s\:{\rm times}},q,q,q,q,\cdots).\]
Let $x^s:=(x_\alpha^s)_{\alpha\in\Gamma}$.
Now we have the following steps:
\\
{\bf Step 1.} For $s\geq1$, the sequence $(\sigma_\varphi^n(x^s))_{n\geq1}$ is one-to-one: 
	Consider $j>i\geq0$, then
	$i<js+\frac{j(j+1)}{2}+s$, so there exists $t\geq1$ with $i+t=js+\frac{j(j+1)}{2}+s$ moreover 
	\[js+\frac{j(j+1)}{2}+s<\underbrace{js+\frac{j(j+1)}{2}+s+(j-i)}_{j+t}<(j+1)s+\frac{(j+1)(j+2)}{2}\:,\]
	which show $x^s_{\varphi^{i+t}(\theta)}=p$ and $x^s_{\varphi^{j+t}(\theta)}=q$ and:
	\[x^s_{\varphi^{i+t}(\theta)}\neq x^s_{\varphi^{j+t}(\theta)}\:.\tag{*}\]
	Using (*) we have $\sigma_\varphi^i((x_\alpha^s)_{\alpha\in\Gamma})\neq
	\sigma_\varphi^j((x_\alpha^s)_{\alpha\in\Gamma})$, thus 
	$(\sigma_\varphi^n(x^s))_{n\geq1}$ is a one-to-one sequence. 
\\
{\bf Step 2.} $(\sigma_\varphi^n(x^1))_{n\geq1}$, $(\sigma_\varphi^n(x^2))_{n\geq1}$, 
	$(\sigma_\varphi^n(x^3))_{n\geq1}$, ... are pairwise disjoint sequences:
	consider $s\geq r\geq1$ and $i,j\geq0$ with $\sigma_\varphi^i(x^s)=\sigma_\varphi^j(x^r)$.
	Choose $m\geq0$ with $i+m=is+\frac{i(i+1)}{2}+1$, now
	we have:
	\begin{eqnarray*}
	\sigma_\varphi^i(x^s)=\sigma_\varphi^j(x^r) & \Rightarrow &
		(\forall\alpha\in\Gamma\:(x^s_{\varphi^i(\alpha)}=x^r_{\varphi^j(\alpha)})) \\
	& \Rightarrow & (\forall n\geq0\: (x^s_{\varphi^{i+n}(\theta)}=x^r_{\varphi^{j+n}(\theta)})) \\
	& \Rightarrow & (\forall k\geq0\: (x^s_{\varphi^{i+m+k}(\theta)}=x^r_{\varphi^{j+m+k}(\theta)})) \\
	& \Rightarrow & (\forall k\in\{0,\ldots,s-1\}\:(p=x^s_{\varphi^{i+m+k}(\theta)}=
	x^r_{\varphi^{j+m+k}(\theta)})) 
	\end{eqnarray*}
	using $x^r_{\varphi^{j+m}(\theta)}=x^r_{\varphi^{j+m+1}(\theta)}=\cdots=
	x^r_{\varphi^{j+m+{s-1}}(\theta)}=p$ and the way of definition of $x^r$ we have $s\leq r$,
	thus $s=r$, and $\sigma_\varphi^i(x^s)=\sigma_\varphi^j(x^s)$ which leads to
$i=j$ by Step 1.
\\
Using the above two steps 
$(\sigma_\varphi^n(x^1))_{n\geq1}$, $(\sigma_\varphi^n(x^2))_{n\geq1}$, 
	$(\sigma_\varphi^n(x^3))_{n\geq1}$, ... are pairwise disjoint infinite sequences
	which leads to $\mathfrak{o}(\sigma_\varphi)=\infty$.
\end{proof}
\begin{lemma}\label{salam30}
Let $W(\varphi)=\varnothing$, and
$\varphi$ is not quasi-periodic,
then $\mathfrak{o}(\sigma_\varphi)=\infty$.
\end{lemma}
\begin{proof}
Since $W(\varphi)=\varnothing$, for all $\alpha\in\Gamma$, $\{\varphi^n(\alpha):n\geq0\}$
is finite. Since $\varphi$ is not quasi-periodic we have 
$\sup\left\{|\{\varphi^n(\alpha):n\geq0\}|:\alpha\in\Gamma\right\}=\infty$.
Thus there exist $\theta_1,\theta_2,\ldots\in\Gamma$ such that for all $i\geq 1$ the set
$\{\theta_i,\varphi(\theta_i),\ldots,\varphi^i(\theta_i)\}$ has $i+1$ elements, moreover 
for all $j\neq i$ we have $\{\theta_i,\varphi(\theta_i),\ldots,\varphi^i(\theta_i)\}\cap 
\{\theta_j,\varphi(\theta_j),\ldots,\varphi^j(\theta_j)\}=\varnothing$. For $n\geq1$ suppose
$u_n$ is the $n$th prime number and choose
distinct $p,q\in X$, now for $m\geq1$ let:
\[x^m_\alpha=\left\{\begin{array}{lc} p & \alpha\in\{\varphi^n(\theta_{u_m^t}):t\geq1,1\leq n< u_m^t\}\:,\\
q & {\rm otherwise}\:,\end{array}\right.\]
so:
{\small
\[\begin{array}{rcl}
 (p,q) & = & (x^1_{\varphi(\theta_2)},x^1_{\varphi^2(\theta_2)}) \\
(p,p,p,q) & = & (x^1_{\varphi(\theta_4)},x^1_{\varphi^2(\theta_4)},x^1_{\varphi^3(\theta_4)},
	x^1_{\varphi^4(\theta_4)}) \\
(p,p,p,p,p,p,p,q) 	& = & (x^1_{\varphi(\theta_8)},x^1_{\varphi^2(\theta_8)},x^1_{\varphi^3(\theta_8)},
	x^1_{\varphi^4(\theta_8)},x^1_{\varphi^5(\theta_8)},x^1_{\varphi^6(\theta_8)},x^1_{\varphi^7(\theta_8)}
	,x^1_{\varphi^8(\theta_8)}) \\
& \vdots & \\
(p,p,q) & = & (x^2_{\varphi(\theta_3)},x^2_{\varphi^2(\theta_3)},x^2_{\varphi^3(\theta_3)}) \\
(p,p,p,p,p,p,p,p,q) 	& = & (x^2_{\varphi(\theta_9)},x^2_{\varphi^2(\theta_9)},x^2_{\varphi^3(\theta_9)},
	x^2_{\varphi^4(\theta_9)},x^2_{\varphi^5(\theta_9)},x^2_{\varphi^6(\theta_9)},x^2_{\varphi^7(\theta_9)}
	,x^2_{\varphi^8(\theta_9)},x^2_{\varphi^9(\theta_9)}) \\
& \vdots & \\
(p,p,p,p,q) 	& = & (x^3_{\varphi(\theta_5)},x^3_{\varphi^2(\theta_5)},x^3_{\varphi^3(\theta_5)},
	x^3_{\varphi^4(\theta_5)},x^3_{\varphi^5(\theta_5)}) \\
& \vdots &  \\
(\underbrace{p,\cdots,p}_{u_m-1{\rm \: times}},q) & = &  (x^m_{\varphi(\theta_{u_m})},
	x^m_{\varphi^2(\theta_{u_m})},\cdots,x^m_{\varphi^{u_m}(\theta_{u_m})}) \\
(\underbrace{p,\cdots,p}_{u_m^2-1{\rm \: times}},q) & = & 
	 (x^m_{\varphi(\theta_{u^2_m})},x^m_{\varphi^2(\theta_{u^2_m})},
	 \cdots,x^m_{\varphi^{u^2_m}(\theta_{u^2_m})}) \\
(\underbrace{p,\cdots,p}_{u_m^3-1{\rm \: times}},q) & = &  (x^m_{\varphi(\theta_{u^3_m})},
	x^m_{\varphi^2(\theta_{u^3_m})}, \cdots,x^m_{\varphi^{u^3_m}(\theta_{u^3_m})}) \\
& \vdots & 
\end{array}\]}
For $m\geq1$, let $x^m:=(x_\alpha^m)_{\alpha\in\Gamma}$. Now we have the following steps:
\\
{\bf Step 1.} For $m\geq1$, the sequence $(\sigma_\varphi^n(x^m))_{n\geq1}$ is one-to-one: 
	Consider $j\geq i\geq1$ with $\sigma_\varphi^i(x^m)=\sigma_\varphi^j(x^m)$, choose
	$t,l\geq1$ such that $j+l=u_m^t$, so:
	\begin{eqnarray*}
	\sigma_\varphi^i(x^m)=\sigma_\varphi^j(x^m) & \Rightarrow &
		(\forall\alpha\in\Gamma\:(x^m_{\varphi^i(\alpha)}=x^m_{\varphi^j(\alpha)})) \\
	& \Rightarrow & (\forall k\geq0\:\forall s\geq1\: 
		(x^m_{\varphi^{i+k}(\theta_{u_m^s})}=x^m_{\varphi^{j+k}(\theta_{u_m^s})})) \\
	& \Rightarrow & (x^m_{\varphi^{i+l}(\theta_{u_m^t})}=
		x^m_{\varphi^{j+l}(\theta_{u_m^t})}=x^m_{\varphi^{u_m^t}(\theta_{u_m^t})}=q) 
	\end{eqnarray*}
	using $x^m_{\varphi^{i+l}(\theta_{u_m^t})}=q$ and $1\leq i+l\leq j+l=u_m^t$
	considering the way of definition of $x^m$ we have $i+l=u_m^t=j+l$, thus $i=j$ and  
	the sequence $(\sigma_\varphi^n(x^m))_{n\geq1}$ is one-to-one.
\\
{\bf Step 2.} $(\sigma_\varphi^n(x^1))_{n\geq1}$, $(\sigma_\varphi^n(x^2))_{n\geq1}$, 
	$(\sigma_\varphi^n(x^3))_{n\geq1}$, ... are pairwise disjoint sequences:
	Consider $r,m\geq1$ and $i\geq j\geq1$ with $\sigma_\varphi^i(x^m)=\sigma_\varphi^j(x^r)$. 
	Choose $l,t\geq1$ with $i+l=u_m^t-1$, so:
	\begin{eqnarray*}
	\sigma_\varphi^i(x^m)=\sigma_\varphi^j(x^r) & \Rightarrow &
		(\forall\alpha\in\Gamma\:(x^m_{\varphi^i(\alpha)}=x^r_{\varphi^j(\alpha)})) \\
	& \Rightarrow & (\forall k\geq0\:\forall s\geq1\: 
		(x^m_{\varphi^{i+k}(\theta_{u_m^s})}=x^m_{\varphi^{j+k}(\theta_{u_r^s})})) \\
	& \Rightarrow & p=x^m_{\varphi^{i+l}(\theta_{u_m^t})}=	x^r_{\varphi^{j+l}(\theta_{u_m^t})}  \\
	& \Rightarrow & \varphi^{j+l}(\theta_{u_m^t})\in\{\varphi^v(u_r^w):w\geq1,1\leq v<u_r^w\} \\
	& \Rightarrow & (\exists w\geq1\:(\varphi^{j+l}(\theta_{u_m^t})\in
		\{\varphi^v(u_r^w):1\leq v< u_r^w\})) \\
	& \Rightarrow &  (\exists w\geq1\:\{\varphi^v(\theta_{u_m^t}):1\leq v< u_m^t\}\cap 
		\{\varphi^v(u_r^w):1\leq v< u_r^w\}) \\
	& & \SP\SP\SP\SP\SP\SP\SP\SP\SP\SP({\rm since \:} 1\leq j+l\leq i+l<u_m^t) \\
	& \Rightarrow &  (\exists w\geq1\:u_m^t=u_r^w)\:({\rm use \: the \: way \: of \: choosing \:}\theta_v{\rm s}) \\
	& \Rightarrow & u_m=u_r \: ({\rm since \:}u_m{\rm \: and \:}u_r{\rm \: are \: prime \: numbers}) \\
	& \Rightarrow & m=r
	\end{eqnarray*}
	thus  $m=r$
	and $\sigma_\varphi^i(x^m)=\sigma_\varphi^j(x^m)$ which leads to
$i=j$ by Step 1.
\\
By the above two steps 
$(\sigma_\varphi^n(x^1))_{n\geq1}$, $(\sigma_\varphi^n(x^2))_{n\geq1}$, 
	$(\sigma_\varphi^n(x^3))_{n\geq1}$, ... are pairwise disjoint infinite sequences
	which leads to $\mathfrak{o}(\sigma_\varphi)=\infty$.
\end{proof}
\begin{theorem}\label{salam40}
The following statements are equivalent:
\begin{itemize}
\item[1.] $W(\sigma_\varphi)=\varnothing$ (i.e., $
\mathfrak{o}(\sigma_\varphi)={\rm ent}_{\rm set}( \sigma_\varphi)=0$),
\item[2.] $\varphi$ is quasi-periodic,
\item[3.] ${\rm ent}_{\rm set}( \sigma_\varphi)<\infty$ (i.e., $\mathfrak{o}(\sigma_\varphi)<\infty$).
\end{itemize}
\end{theorem}
\begin{proof}
``(1) $\Rightarrow$ (2)'':
Suppose $W(\sigma_\varphi)=\varnothing$, thus by Lemma~\ref{salam10}  we have
$W(\varphi)=\varnothing$. Since $W(\sigma_\varphi)=\varnothing$, we have
$\mathfrak{o}(\sigma_\varphi)=0$. Using $\mathfrak{o}(\sigma_\varphi)=0$,
$W(\varphi)=\varnothing$ and Lemma~\ref{salam30}, $\varphi$ is quasi-periodic.
\\
``(2) $\Rightarrow$ (1)'':
If there exist $n>m\geq1$ with $\varphi^n=\varphi^m$, then for all
$(x_\alpha)_{\alpha\in\Gamma}\in X^\Gamma$ we have
$\sigma_\varphi^n((x_\alpha)_{\alpha\in\Gamma})=
(x_{\varphi^n(\alpha)})_{\alpha\in\Gamma}=(x_{\varphi^m(\alpha)})_{\alpha\in\Gamma}=
\sigma_\varphi^m((x_\alpha)_{\alpha\in\Gamma})$ which shows 
$(x_\alpha)_{\alpha\in\Gamma}\notin W(\sigma_\varphi)$.
\\
``(3) $\Rightarrow$ (2)'': By Lemmas~\ref{salam20} and \ref{salam30}, if $\varphi$ is not quasi-periodic,
then $\mathfrak{o}(\sigma_\varphi)=\infty$.
\end{proof}
\begin{corollary}\label{salam45}
By Theorem~\ref{salam40} we have:
\[{\rm ent}_{\rm set}( \sigma_\varphi)=\left\{\begin{array}{lc} 0 & \varphi{\rm \: is \: quasi-periodic}\:,\\
\infty & {\rm otherwise}\:.\end{array}\right.\]
\end{corollary}
\section{Contravariant set-theoretical entropy of $\sigma_\varphi:X^\Gamma\to X^\Gamma$}
\noindent For $\alpha,\beta\in\Gamma$ let $\alpha\Re\beta$ if and only if there exists $n\geq1$
with $\varphi^n(\alpha)=\varphi^n(\beta)$. Then $\Re$ is an equivalence relation on 
$\Gamma$, moreover it's evident that $\alpha\Re\beta$ if and only if $\varphi(\alpha)\Re\varphi(\beta)$.
In this section we prove that for all $x\in X^\Gamma$, $\sigma_\varphi^{-1}(x)$ is finite,
if and only if either $\Gamma=\varphi(\Gamma)$ or ``$X$ and 
$\Gamma\setminus\varphi(\Gamma)$ are finite''.
Moreover if for all $x\in X^\Gamma$, $\sigma_\varphi^{-1}(x)$ is finite, then 
${\rm ent}_{\rm cset}( \sigma_\varphi)\in\{0,\infty\}$ with ${\rm ent}_{\rm cset}( \sigma_\varphi)=0$
if and only if there exists $n\geq1$ such that $\varphi^n(\alpha)\Re\alpha$ for all
$\alpha\in\Gamma$.
\begin{remark}\label{salam5}
The generalized shift $\sigma_\varphi:X^\Gamma\to X^\Gamma$ is
on-to-one (resp. onto) if and only if $\varphi:\Gamma\to\Gamma$ is onto
(resp. one-to-one) \cite{AHK, set}.
\end{remark}
\begin{note}\label{salam50}
Consider $\widetilde{\varphi}:\mathop{\frac{\Gamma}{\Re}\to\frac{\Gamma}{\Re}}\limits_{[\alpha]_\Re\mapsto[\varphi(\alpha)]_\Re}$ and note that 
(for $\alpha\in\Gamma$ let $[\alpha]_\Re=\{y\in\Gamma:x\Re y\}$ and $\frac{\Gamma}{\Re}=\{[\lambda]_\Re:\lambda\in\Gamma\}$):
\[\mathop{\mathfrak{f}:{\rm sc}(\sigma_\varphi)\to X^{\frac{\Gamma}{\Re}}}\limits_{\SP\SP\SP\SP
(x_\alpha)_{\alpha\in\Gamma}\mapsto(x_\alpha)_{[\alpha]_\Re\in\frac{\Gamma}{\Re}}}\]
is well-defined, since for $(x_\alpha)_{\alpha\in\Gamma}\in{\rm sc}(\sigma_\varphi)$
and $\theta,\beta\in\Gamma$ with $\theta\Re\beta$, there exists $n\geq1$ and 
$(y_\alpha)_{\alpha\in\Gamma}\in X^\Gamma$ with $\varphi^n(\theta)=\varphi^n(\beta)$
and $(y_{\varphi^n(\alpha)})_{\alpha\in\Gamma}=\sigma_\varphi^n((y_\alpha)_{\alpha\in\Gamma})
=(x_\alpha)_{\alpha\in\Gamma}$, thus $x_\theta=y_{\varphi^n(\theta)}=y_{\varphi^n(\beta)}=x_\beta$.
Now we have:
\\
1. $\mathfrak{f}:{\rm sc}(\sigma_\varphi)\to X^{\frac{\Gamma}{\Re}}$ is one-to-one.
\\
2. The following diagram commutes:
\[\xymatrix{{\rm sc}(\sigma_\varphi) \ar[rr]^{\sigma_\varphi\restriction_{{\rm sc}(\sigma_\varphi)}}
\ar[d]_{\mathfrak{f}} && {\rm sc}(\sigma_\varphi) \ar[d]^{\mathfrak{f}} \\
X^{\frac{\Gamma}{\Re}} \ar[rr]^{\sigma_{\widetilde{\varphi}}} && X^{\frac{\Gamma}{\Re}}}
\]
3. Using  Remark~\ref{salam5}, since $\widetilde{\varphi}:\frac{\Gamma}{\Re}\to\frac{\Gamma}{\Re}$ is one-to-one, $\sigma_{\widetilde{\varphi}}:X^{\frac{\Gamma}{\Re}}\to X^{\frac{\Gamma}{\Re}}$ is onto.
\end{note}
\begin{lemma}\label{salam55}
The generalized shift $\sigma_\varphi:X^\Gamma\to X^\Gamma$  is finite fibre,
if and only if at least one of the following conditions hold:
\begin{itemize}
\item $X$ and $\Gamma\setminus\varphi(\Gamma)$ are finite,
\item $\Gamma=\varphi(\Gamma)$.
\end{itemize}
\end{lemma}
\begin{proof}
First suppose for all $x\in X^\Gamma$, $\sigma_\varphi^{-1}(x)$ is 
finite and $\Gamma\neq\varphi(\Gamma)$. Choose $p\in X$, for all 
$q=(q_\alpha)_{\alpha\in\Gamma\setminus\varphi(\Gamma)}\in X^{\Gamma\setminus\varphi(\Gamma)}$
let:
\[x_\alpha^q:=\left\{\begin{array}{lc} q_\alpha & \alpha\in\Gamma\setminus\varphi(\Gamma)\:, \\
p & {\rm otherwise}\:,\end{array}\right.\]
then
$\sigma_\varphi((x_\alpha^q)_{\alpha\in\Gamma})=(p)_{\alpha\in\Gamma}$.
So $\mathop{X^{\Gamma\setminus\varphi(\Gamma)}\to\sigma_\varphi^{-1}((p)_{\alpha\in\Gamma})}\limits_{q\mapsto
(x_\alpha^q)_{\alpha\in\Gamma}}$ is one-to-one, using finiteness of 
$\sigma_\varphi^{-1}((p)_{\alpha\in\Gamma})$, $X^{\Gamma\setminus\varphi(\Gamma)}$ is finite too. Both sets $\Gamma\setminus\varphi(\Gamma),X$ are finite
since $X^{\Gamma\setminus\varphi(\Gamma)}$ is finite, $X$ has at least two elements and $\Gamma\setminus\varphi(\Gamma)\neq
\varnothing$.
\\
Conversely, if $\Gamma=\varphi(\Gamma)$, then by Remark~\ref{salam5}, $\sigma_\varphi:X^\Gamma\to X^\Gamma$
is one-to-one, so for all $x\in X^\Gamma$ the set $\sigma_\varphi^{-1}(x)$  has at most one element
and is finite. Now suppose $X$ and $\Gamma\setminus\varphi(\Gamma)$ are finite. For
all $(x_\alpha)_{\alpha\in\Gamma},(y_\alpha)_{\alpha\in\Gamma}\in X^\Gamma$ we have:
\begin{eqnarray*}
(y_\alpha)_{\alpha\in\Gamma}\in\sigma_\varphi^{-1}(\sigma_\varphi((x_\alpha)_{\alpha\in\Gamma}))	
	& \Rightarrow & \sigma_\varphi((y_\alpha)_{\alpha\in\Gamma})
	=\sigma_\varphi((x_\alpha)_{\alpha\in\Gamma}) \\
& \Rightarrow & (y_{\varphi(\alpha)})_{\alpha\in\Gamma}=(x_{\varphi(\alpha)})_{\alpha\in\Gamma} \\
& \Rightarrow & \forall\beta\in\varphi(\Gamma)\: y_\beta=x_\beta \\
& \Rightarrow & (y_\alpha)_{\alpha\in\Gamma}\in\{(z_\alpha)_{\alpha\in\Gamma}\in X^\Gamma:
	\forall\alpha\in\varphi(\Gamma)\:z_\alpha=x_\alpha\}
\end{eqnarray*}
Hence 
\[|\sigma_\varphi^{-1}(\sigma_\varphi((x_\alpha)_{\alpha\in\Gamma}))|\leq|
\{(z_\alpha)_{\alpha\in\Gamma}\in X^\Gamma:
	\forall\alpha\in\varphi(\Gamma)\:z_\alpha=x_\alpha\}|=|X^{\Gamma\setminus\varphi(\Gamma)}|<\infty\:.\]
Thus for all $w\in X^\Gamma$, $\sigma_\varphi^{-1}(w)$ is finite.
\end{proof}
\begin{lemma}\label{salam56}
If $\sigma_\varphi:X^\Gamma\to X^\Gamma$ is finite fibre,
then $\sigma_{\widetilde{\varphi}}:X^{\frac{\Gamma}{\Re}}\to X^{\frac{\Gamma}{\Re}}$ is finite fibre.
\end{lemma}
\begin{proof}
Suppose for all $w\in X^\Gamma$, $\sigma_\varphi^{-1}(w)$ is finite, then by Lemma~\ref{salam55} we have the following cases:
\\
{\bf Case 1.} $\Gamma=\varphi(\Gamma)$: In this case we have
$\widetilde{\varphi}(\frac{\Gamma}{\Re})
=\{[\varphi(\alpha)]_\Re:\alpha\in\Gamma\}=\{[\alpha]_\Re:\alpha\in\varphi(\Gamma)\}=
\{[\alpha]_\Re:\alpha\in\Gamma\}=\frac{\Gamma}{\Re}$.
\\
{\bf Case 2.} $X$ and $\Gamma\setminus\varphi(\Gamma)$ are finite: For all $A\in\frac{\Gamma}{\Re}\setminus\widetilde{\varphi}(\frac{\Gamma}{\Re})$ we have $A\subseteq \Gamma\setminus\varphi(\Gamma)$ which leads to $|\frac{\Gamma}{\Re}\setminus\widetilde{\varphi}(\frac{\Gamma}{\Re})|\leq|\bigcup(\frac{\Gamma}{\Re}\setminus\widetilde{\varphi}(\frac{\Gamma}{\Re})|\leq|\Gamma\setminus\varphi(\Gamma)|$ and $\frac{\Gamma}{\Re}\setminus\widetilde{\varphi}(\frac{\Gamma}{\Re})$ is finite in this case.
\\
Using the above two cases and Lemma~\ref{salam55}, 
$\sigma_{\widetilde{\varphi}}$ is finite fibre.
\end{proof}
\begin{lemma}\label{salam57}
We have
$\mathfrak{a}( \sigma_\varphi)=\mathfrak{a}( \sigma_{\widetilde{\varphi}})$. 
In particular by Lemma~\ref{salam56}, if $\sigma_\varphi:X^\Gamma\to X^\Gamma$ is finite fibre,
then 
${\rm ent}_{\rm cset}( \sigma_\varphi)={\rm ent}_{\rm cset}( \sigma_{\widetilde{\varphi}})$.
\end{lemma}
\begin{proof}
For $s\geq1$ suppose $(y_n^1)_{n\geq1},\ldots,(y_n^s)_{n\geq1}$ are pairwise
disjoint infinite $\sigma_\varphi-$anti-orbit sequences, then for all $i\in\{1,\ldots,s\}$ and
$n\geq1$ we have $y^i_n\in{\rm sc}(\sigma_\varphi)$. Using Note~\ref{salam50}, $\mathfrak f$
is one-to-one, thus $(\mathfrak{f}(y_n^1))_{n\geq1},\ldots,(\mathfrak{f}(y_n^s))_{n\geq1}$
are infinite pairwise disjoint sequences in $X^{\frac{\Gamma}{\Re}}$, moreover
 for all $i\in\{1,\ldots,s\}$ and
$n\geq1$ we have $\sigma_{\widetilde{\varphi}}(\mathfrak{f}(y_{n+1}^i))=\mathfrak{f}(\sigma_\varphi(y_{n+1}^i))=\mathfrak{f}(y_n^i)$. Thus  $(\mathfrak{f}(y_n^1))_{n\geq1},\ldots,(\mathfrak{f}(y_n^s))_{n\geq1}$
are infinite pairwise disjoint $\sigma_{\widetilde{\varphi}}-$anti-orbit sequences.
Therefore $\mathfrak{a}( \sigma_{\widetilde{\varphi}})\geq \mathfrak{a}( \sigma_\varphi)$.
\\
Now for $x=(x_\beta)_{\beta\in\frac{\Gamma}{\Re}}\in X^{\frac{\Gamma}{\Re}}$ let $w^x=(x_{[\alpha]_\Re})_{\alpha\in\Gamma}\in X^\Gamma$. 
So for all $x,y\in X^{\frac{\Gamma}{\Re}}$
if $w^x=w^y$, then $x=y$. Moreover for $x=(x_\beta)_{\beta\in\frac{\Gamma}{\Re}}\in X^{\frac{\Gamma}{\Re}}$ we have:
\begin{eqnarray*}
w^{\sigma_{\widetilde{\varphi}}(x)} & = &
	w^{(x_{\widetilde{\varphi}(\beta)})_{\beta\in\frac{\Gamma}{\Re}}}=
	(x_{\widetilde{\varphi}	
	([\alpha]_\Re)})_{\alpha\in\Gamma} \\
& = & (x_{[\varphi(\alpha)]_\Re})_{\alpha\in\Gamma}=
\sigma_\varphi((x_{[\alpha]_\Re})_{\alpha\in\Gamma})=\sigma_\varphi(w^x)\:.
\end{eqnarray*}
For $t\geq1$ suppose $(y_n^1)_{n\geq1},\ldots,(y_n^t)_{n\geq1}$ are pairwise
disjoint infinite $\sigma_{\widetilde{\varphi}}-$anti-orbit sequences, then 
$(w^{y_n^1})_{n\geq1},\ldots,(w^{y_n^t})_{n\geq1}$ are pairwise
disjoint infinite  $\sigma_\varphi-$anti-orbit sequences.
Thus $\mathfrak{a}( \sigma_\varphi)\geq\mathfrak{a}( \sigma_{\widetilde{\varphi}})$. 
\end{proof}
\begin{lemma}\label{salam60}
Suppose $\psi:\Gamma\to\Gamma$ is one-to-one and has at least one non-periodic point,
then $\mathfrak{a}(\sigma_\psi)=\infty$, thus if 
$\sigma_\psi:X^\Gamma\to X^\Gamma$ is finite fibre too, then
${\rm ent}_{\rm cset}( \sigma_\psi)=\infty$.
\end{lemma} 
\begin{proof}
Suppose $\varphi:\Gamma\to\Gamma$ is one-to-one, and
$\theta\in\Gamma$ is a non-periodic point of $\varphi$.
Choose distinct $p,q\in X$ and for $m,n\geq 1$ let:
\[(x_n^m(0),x_n^m(1),x_n^m(2),\cdots):=(\:\underbrace{p,\cdots,p}_{m {\rm \: times}},
\underbrace{q,\cdots,q}_{n {\rm \: times}}, p,p,p,\cdots)\:,\]
now let:
\[z^{(n,m)}_\alpha:=\left\{\begin{array}{lc} x_n^m(k) & k\geq0,\alpha=\varphi^k(\theta) \: , \\
p & {\rm otherwise}\:. \end{array}\right. \]
Then for $z^{(n,m)}:=(z^{(n,m)}_\alpha)_{\alpha\in\Gamma}$, considering the sequences
\[(z^{(1,m)})_{m\geq1},(z^{(2,m)})_{m\geq1},(z^{(3,m)})_{m\geq1},\ldots\]
we have:
\begin{itemize}
\item For $k,n,i,j\geq1$ if $z^{(n,i)}=z^{(k,j)}$, then:
	\begin{eqnarray*}
	\SP\SP\SP z^{(n,i)}=z^{(k,j)} & \Rightarrow &(\forall\alpha\in\Gamma\: z_\alpha^{(n,i)}=z_\alpha^{(k,j)}) \\
	& \Rightarrow & (z_\theta^{(n,i)},z_{\varphi(\theta)}^{(n,i)},z_{\varphi^2(\theta)}^{(n,i)},\cdots)
		(z_\theta^{(k,j)},z_{\varphi(\theta)}^{(k,j)},z_{\varphi^2(\theta)}^{(k,j)},\cdots) \\
	& \Rightarrow & (x_n^i(0),x_n^i(1),\cdots)=(x_k^j(0),x_k^j(1),\cdots) \\
	& \Rightarrow & (\:\underbrace{p,\cdots,p}_{i {\rm \: times}},
\underbrace{q,\cdots,q}_{n {\rm \: times}}, p,\cdots)=(\:\underbrace{p,\cdots,p}_{j {\rm \: times}},
\underbrace{q,\cdots,q}_{k {\rm \: times}}, p,\cdots) \\
& \Rightarrow & (i=j\wedge n=k)
	\end{eqnarray*}
	Thus $(z^{(1,m)})_{m\geq1},(z^{(2,m)})_{m\geq1},(z^{(3,m)})_{m\geq1},\ldots$
	are paiwise disjoint infinite sequences.
\item For all $n,m\geq1$ and $\alpha\in \Gamma$ we have:
\begin{eqnarray*}
z^{(n,m+1)}_{\varphi(\alpha)}=q & \Leftrightarrow & \varphi(\alpha)\in\{\varphi^i(\theta):m+1\leq i<m+1+n\} \\
& \Leftrightarrow &\alpha\in\{\varphi^i(\theta):m\leq i<m+n\} \\
&\Leftrightarrow & z^{(n,m)}_\alpha=q
\end{eqnarray*}
thus $\sigma_\varphi(z^{(n,m+1)})=z^{(n,m)}$ and $(z^{(n,k)})_{k\geq1}$
is an anti-orbit
\end{itemize}
Hence $(z^{(1,m)})_{m\geq1},(z^{(2,m)})_{m\geq1},(z^{(3,m)})_{m\geq1},\ldots$ are pairwise
disjoint infinite $\sigma_\varphi-$anti-orbit sequences and $\mathfrak{a}(\sigma_\varphi)=\infty$.
\end{proof}
\begin{note}\label{salam70}
Suppose all points of $\Gamma$ are periodic points of $\varphi:\Gamma\to\Gamma$,
then $\sigma_\varphi:X^\Gamma\to X^\Gamma$ is bijective
(note that $\varphi:\Gamma\to\Gamma$ is bijective and apply Remark~\ref{salam5})
and using Corollary~\ref{salam45} we have:
\[{\rm ent}_{\rm cset}( \sigma_\varphi)={\rm ent}_{\rm set}( \sigma_\varphi^{-1})=
{\rm ent}_{\rm set}( \sigma_{\varphi^{-1}})=\left\{\begin{array}{lc} 0 & \exists n\geq1\:\varphi^n={\rm id}_\Gamma
\: , \\ \infty & {\rm otherwise}\:.\end{array}\right.\]
where for arbitrary $A$ we have $\mathop{{\rm id}_A:A\to A}\limits_{\SP\SP\:\:x\mapsto x}$.
\end{note}
\begin{corollary}\label{salam80}
Suppose $\varphi:\Gamma\to\Gamma$ is one-to-one and 
$\sigma_\varphi:X^\Gamma\to X^\Gamma$ is finite fibre, then
\[{\rm ent}_{\rm cset}( \sigma_\varphi)={\rm ent}_{\rm set}( \sigma_\varphi)
=\left\{\begin{array}{lc} 0 & \exists n\geq1\:\varphi^n={\rm id}_\Gamma
\: , \\ \infty & {\rm otherwise}\:.\end{array}\right.\]
\end{corollary} 
\begin{proof}
Use Corollary~\ref{salam45}, Lemma~\ref{salam60} and Note~\ref{salam70}.
\end{proof}
\begin{corollary}\label{salam90}
If $\sigma_\varphi:X^\Gamma\to X^\Gamma$ is finite fibre, then:
\[{\rm ent}_{\rm cset}( \sigma_\varphi)={\rm ent}_{\rm cset}( \sigma_{\widetilde{\varphi}})=\left\{\begin{array}{lc} 0 & \exists n\geq1\:
(\widetilde{\varphi})^n={\rm id}_\frac{\Gamma}{\Re}\: , \\ \infty & {\rm otherwise}\:.\end{array}\right.\]
\end{corollary}
\begin{proof}
First we recall that $\widetilde{\varphi}:\frac{\Gamma}{\Re}\to\frac{\Gamma}{\Re}$ is one-to-one
by Note~\ref{salam50}. Use Corollary~\ref{salam80} and Lemma~\ref{salam57} to complete the proof.
\end{proof}
\section{Other entropies: counterexamples}
\noindent The main aim of this section is to compare positive topological, algebraic, set-theoretical
and contravariant set-theoretical entropies in generalized shifts.
\begin{remark}\label{salam200}
If $G$ is an abelian group, $\theta:G\to G$ is a group homomorphism and $H$ is a
finite subset of $G$, then ${\rm ent}_{\rm alg}(\theta,H)={\displaystyle\lim_{n\to\infty}\dfrac{\log(|
H\cup\theta(H)\cup\cdots\cup\theta^{n-1}(H)|)}{n}}$ exists \cite{anna2, salce} and we call
${\rm ent}_{\rm alg}(\theta):=\sup\{{\rm ent}_{\rm alg}(\theta,H):H$ is a finite subgroup of $G\}$
the {\it algebraic entropy} of $\theta$. Moreover if $\varphi:\Gamma\to\Gamma$ is finite fibre
and $X$ is a finite nontrivial group with identity $e$, then ${\rm ent}_{\rm alg}(\sigma_\varphi\restriction_{
\mathop{\oplus}\limits_{\Gamma}X})={\rm ent}_{\rm cset}(\varphi)\log|X|$ 
(as it has been mentioned in  \cite[Theorem 4.14]{akhavin} ${\rm ent}_{\rm alg}(\sigma_\varphi\restriction_{
\mathop{\oplus}\limits_{\Gamma}X})$ is equal to the product of string number of $\varphi$ and $\log|X|$
this result has been evaluated in \cite[Theorem 7.3.3]{anna2} in the above form), where
$\mathop{\oplus}\limits_{\Gamma}X=\{(x_\alpha)_{\alpha\in\Gamma}\in X^\Gamma:
\exists\alpha_1,\ldots,\alpha_n\in\Gamma\:\forall\alpha\in\Gamma\setminus\{\alpha_1,\ldots,\alpha_n\}
\:(x_\alpha=e)\}$. Also by \cite{anna}, ${\rm ent}_{\rm alg}(\sigma_\varphi)\in\{0,\infty\}$ with
 ${\rm ent}_{\rm alg}(\sigma_\varphi)=0$ if and only if there exists $n>m\geq1$ with
 $\varphi^n=\varphi^m$ (thus ${\rm ent}_{\rm alg}(\sigma_\varphi)={\rm ent}_{\rm set}(\sigma_\varphi)$ by
 Corollary~\ref{salam45}).
 \end{remark}
\begin{remark}
Suppose $Y$ is a compact topological space and $\mathcal{U},\mathcal{V}$ are open
covers of $Y$, let $\mathcal{U}\vee\mathcal{V}:=\{U\cap V:U\in\mathcal{U},V\in\mathcal{V}\}$
and $N(\mathcal{U}):=\min\{|\mathcal{W}|:\mathcal{W}$ is a finite subcover of $\mathcal{U}\}$.
Now suppose $T:Y\to Y$ is continuous, then 
${\rm ent}_{\rm top}(T,\mathcal{U}):=
{\displaystyle\lim_{n\to\infty}\dfrac{\log N(\mathcal{U}\vee T^{-1}(\mathcal{U})\vee\cdots\vee
T^{-(n-1)}(\mathcal{U}))}{n}}$ exists \cite{walters} and we call 
${\rm ent}_{\rm top}(T):=\sup\{{\rm ent}_{\rm top}(T,\mathcal{W}):\mathcal{W}$ is a finite
open cover of $Y\}$ the {\it topological entropy} of $T$. If $X$ is a finite discrete 
topological space with at least two elements and $X^\Gamma$ considered with product (pointwise
convergence) topology, then ${\rm ent}_{\rm top}(\sigma_\varphi)={\rm ent}_{\rm set}(\varphi)\log|X|$
\cite{set}.
\end{remark}
\noindent In the rest let:
\begin{itemize}
\item $\mathcal{C}$ is the collection of all generalized shifts $\sigma_\psi:Y^\Gamma\to
	Y^\Gamma$ such that $Y$ is a nontrivial finite discrete topological group (so $2\leq|Y|<\infty$), $\Gamma$ is a nonempty set and both maps
	$\psi:\Gamma\to\Gamma$, $\sigma_\psi:Y^\Gamma\to
	Y^\Gamma$ are finite fibre,
\item $\mathcal{C}_{\rm top}$ is the collection of all elements of $\mathcal C$ like
	$\sigma_\psi:Y^\Gamma\to Y^\Gamma$ such that 
	${\rm ent}_{\rm top}(\sigma_\psi)>0$,
\item $\mathcal{C}_{\rm dalg}$ is the collection of all elements of $\mathcal C$ like
	$\sigma_\psi:Y^\Gamma\to Y^\Gamma$ such that 
	${\rm ent}_{\rm alg}(\sigma_\psi\restriction_{\mathop{\oplus}\limits_{\Gamma}Y})>0$,
\item $\mathcal{C}_{\rm cset}$ is the collection of all elements of $\mathcal C$ like
	$\sigma_\psi:Y^\Gamma\to Y^\Gamma$ such that
	${\rm ent}_{\rm cset}(\sigma_\psi)>0$,
\item $\mathcal{C}_{\rm set}$ is the collection of all elements of $\mathcal C$ like
	$\sigma_\psi:Y^\Gamma\to Y^\Gamma$ such that $
	{\rm ent}_{\rm set}(\sigma_\psi)>0$ (i.e., ${\rm ent}_{\rm alg}(\sigma_\psi)>0$ by Remark~\ref{salam200}).
\end{itemize}
\begin{lemma}\label{salam95}
We have $\mathcal{C}_{\rm top}\subseteq \mathcal{C}_{\rm cset}\subseteq  \mathcal{C}_{\rm set}$
and $\mathcal{C}_{\rm dalg}\subseteq  \mathcal{C}_{\rm set}$. As a matter of fact
for an element of
$\mathcal C$ like $\sigma_\varphi:Y^\Gamma\to Y^\Gamma$ we have:
\begin{center}
${\rm ent}_{\rm top}(\sigma_\varphi)\leq {\rm ent}_{\rm cset}(\sigma_\varphi)\leq 
{\rm ent}_{\rm set}(\sigma_\varphi)$ and 
${\rm ent}_{\rm alg}(\sigma_\varphi\restriction_{\mathop{\oplus}\limits_{\Gamma}Y})\leq
{\rm ent}_{\rm set}(\sigma_\varphi)$.
\end{center}
\end{lemma}
\begin{proof} $\:$
\begin{itemize}
\item ``${\rm ent}_{\rm top}(\sigma_\varphi)\leq {\rm ent}_{\rm cset}(\sigma_\varphi)$'' 
Suppose  ${\rm ent}_{\rm top}(\sigma_\varphi)>0$, then
\linebreak
$\mathfrak{o}(\varphi)\log|Y|={\rm ent}_{\rm set}(\varphi)\log|Y|={\rm ent}_{\rm top}(\sigma_\varphi)>0$,
thus $\mathfrak{o}(\varphi)>0$ and $W(\varphi)\neq\varnothing$. Choose
$\alpha\in W(\varphi)$, then $\{\varphi^n(\alpha)\}_{n\geq0}$ is a one-to-one sequence
thus for all $n>m\geq0$, $[\varphi^n(\alpha)]_\Re\neq[\varphi^m(\alpha)]_\Re$, so for all
$n\geq1$ we have $\widetilde{\varphi}^n([\alpha]_\Re)\neq[\alpha]_\Re$, hence by
Corollary~\ref{salam90}, ${\rm ent}_{\rm cset}(\sigma_\varphi)>0$ and 
${\rm ent}_{\rm cset}(\sigma_\varphi)=\infty(\geq{\rm ent}_{\rm top}(\sigma_\varphi)$).
\item ``${\rm ent}_{\rm cset}(\sigma_\varphi)\leq {\rm ent}_{\rm set}(\sigma_\varphi)$'' 
Suppose
${\rm ent}_{\rm set}(\sigma_\varphi)\neq\infty$, then ${\rm ent}_{\rm set}(\sigma_\varphi)=0$
and there exists $n>m\geq1$ with $\varphi^n=\varphi^m$, thus $\widetilde{\varphi}^n=
\widetilde{\varphi}^m$, and using the fact that $\widetilde{\varphi}$ is one-to-one
we lave $\widetilde{\varphi}^{n-m}={\rm id}_{\frac{\Gamma}{\Re}}$, thus
${\rm ent}_{\rm cset}(\sigma_\varphi)=0$ by Corollary~\ref{salam90}.
\item ``${\rm ent}_{\rm alg}(\sigma_\varphi\restriction_{\mathop{\oplus}\limits_{\Gamma}Y})\leq
{\rm ent}_{\rm set}(\sigma_\varphi)$''
Suppose  
${\rm ent}_{\rm alg}(\sigma_\varphi\restriction_{\mathop{\oplus}\limits_{\Gamma}Y})>0$, then
\linebreak
$\mathfrak{a}(\varphi)\log|Y|={\rm ent}_{\rm cset}(\varphi)\log|Y|={\rm ent}_{\rm alg}(\sigma_\varphi\restriction_{\mathop{\oplus}\limits_{\Gamma}Y})>0$,
thus $\mathfrak{a}(\varphi)>0$ and there exists a one-to-one anti-orbit sequence
$\{\alpha_n\}_{n\geq1}$ in $\Gamma$.
For all $n>m\geq1$ we have 
$\varphi^n(\alpha_{n+m})=\alpha_m\neq\alpha_n=\varphi^m(\alpha_{n+m})$ and
$\varphi^n\neq\varphi^m$, thus 
${\rm ent}_{\rm set}(\sigma_\varphi)=\infty(\geq{\rm ent}_{\rm alg}(\sigma_\varphi\restriction_{\mathop{\oplus}\limits_{\Gamma}Y})$) by Corollary~\ref{salam45}.
\end{itemize}
\end{proof}
\begin{jadval}\label{salam99}
We have the following table, in which the mark ``$\surd$'' means
$p\leq q$ for the corresponding case for all $\sigma_\psi:Y^\Gamma\to Y^\Gamma$
in $\mathcal C$, also the mark ``$\times$'' indicates that there exists 
$\sigma_\psi:Y^\Gamma\to Y^\Gamma$
in $\mathcal C$ with $p>q$ in the corresponding case.
\begin{center}
\begin{tabular}{l|c|c|c|c|}
$\dfrac{\SP\SP\SP\SP\SP\SP q}{p\SP\SP\SP\SP\SP\SP}$ & ${\rm ent}_{\rm top}(\sigma_\varphi)$ & 
${\rm ent}_{\rm alg}(\sigma_\varphi\restriction_{\mathop{\oplus}\limits_{\Gamma}Y})$ &
${\rm ent}_{\rm cset}(\sigma_\varphi)$ & ${\rm ent}_{\rm set}(\sigma_\varphi)$ \\ \hline
${\rm ent}_{\rm top}(\sigma_\varphi)$ & $\surd$ & $\times$ & $\surd$ & $\surd$ \\ \hline
${\rm ent}_{\rm alg}(\sigma_\varphi\restriction_{\mathop{\oplus}\limits_{\Gamma}Y})$ &
$\times$ & $\surd$ & $\times$ & $\surd$ \\ \hline 
${\rm ent}_{\rm cset}(\sigma_\varphi)$ & $\times$ & $\times$ & $\surd$ & $\surd$ \\ \hline
${\rm ent}_{\rm set}(\sigma_\varphi)$ & $\times$ & $\times$ & $\times$ & $\surd$ \\ \hline
\end{tabular}
\end{center}
\end{jadval}
\begin{proof}
For all ``$\surd$'' marks use Lemma~\ref{salam95}. In order to establish ``$\times$'' marks
use the following counterexamples.
\\
Define $\lambda_1,\lambda_2,\lambda_3:{\mathbb Z}\to{\mathbb Z}$ with the following diagrams:
{\small \begin{center}
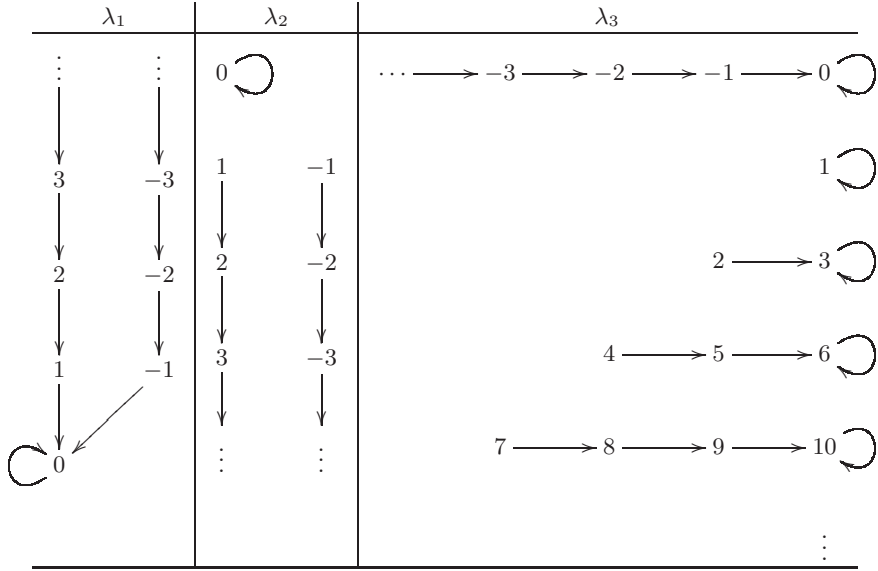
\begin{tabular}{c|c|c}
$\lambda_1$ & $\lambda_2$ & $\lambda_3$ \\ \hline 
	$\xymatrix{ \vdots\ar[d] & \vdots\ar[d] \\ 
	3 \ar[d] & -3 \ar[d] \\
	2 \ar[d] & -2 \ar[d] \\
	1 \ar[d] & -1 \ar[dl] \\
	0\ar@(dl,lu) & }$
& $\xymatrix{0 \ar@(ur,rd) & \\ 1 \ar[d] & -1 \ar[d] \\ 2 \ar[d] & -2 \ar[d] \\ 3 \ar[d] & -3 \ar[d] \\
\vdots & \vdots}$
& $\xymatrix{\cdots \ar[r] & -3 \ar[r] & -2 \ar[r] & -1 \ar[r] & 0  \ar@(ur,rd) \\
  &   &   &   & 1  \ar@(ur,rd) \\
 &  &  & 2 \ar[r] & 3  \ar@(ur,rd) \\
 &  & 4 \ar[r] & 5 \ar[r] & 6  \ar@(ur,rd) \\
& 7 \ar[r] & 8 \ar[r] & 9 \ar[r] & 10  \ar@(ur,rd) \\
 &  &  &  & \vdots }$ \\ \hline
\end{tabular}
\end{center}}
\noindent So:
{\small
\[\lambda_1(n)=\left\{\begin{array}{lc} n-1 & n\geq1\:, \\ 0 & n=0\:, \\ n+1 & n\leq-1\:, \end{array}
\right.
\SP
\lambda_2(n)=\left\{\begin{array}{lc} n+1 & n\geq1\:, \\ 0 & n=0\:, \\ n-1 & n\leq-1\:, \end{array}
\right.\SP 
\lambda_3(n)=\left\{\begin{array}{lc} n+1 & n\leq-1\:, \\ 0 & n=0,1\:, \\ 3 & n=2\:, \\
3 & n=3 \:, \\ 5 & n=4 \:, \\ 6 & n=5 \: \\ 6 & n=6\: , \\  \vdots & \end{array}
\right.\]}
\noindent Then for discrete finite abelian group $G$ with $|G|\geq2$ and
$\sigma_{\lambda_i}:G^{\mathbb Z}\to G^{\mathbb Z}$ we have:
\\
$\bullet$ $\mathfrak{o}(\lambda_1)=\mathfrak{o}(\lambda_3)=0$, $\mathfrak{o}(\lambda_2)=2$,
	$\mathfrak{a}(\lambda_1)=2$, $\mathfrak{a}(\lambda_2)=0$, $\mathfrak{a}(\lambda_3)=1$,
\\
$\bullet$ ${\rm ent}_{\rm top}(\sigma_{\lambda_1})={\rm ent}_{\rm top}(\sigma_{\lambda_3})=0$,
	${\rm ent}_{\rm top}(\sigma_{\lambda_2})=2\log|G|$,
\\
$\bullet$ ${\rm ent}_{\rm alg}(\sigma_{\lambda_1}\restriction_{\mathop{\oplus}\limits_{\Gamma}G})
	=2\log|G|$, 
	${\rm ent}_{\rm alg}(\sigma_{\lambda_2}\restriction_{\mathop{\oplus}\limits_{\Gamma}G})=0$,
	${\rm ent}_{\rm alg}(\sigma_{\lambda_3}\restriction_{\mathop{\oplus}\limits_{\Gamma}G})=
	\log|G|$,
\\
$\bullet$ ${\rm ent}_{\rm cset}(\sigma_{\lambda_1})={\rm ent}_{\rm cset}(\sigma_{\lambda_3})=0$,
	${\rm ent}_{\rm cset}(\sigma_{\lambda_2})=\infty$,
\\
$\bullet$ ${\rm ent}_{\rm set}(\sigma_{\lambda_1})={\rm ent}_{\rm set}(\sigma_{\lambda_2})=
	{\rm ent}_{\rm set}(\sigma_{\lambda_3})=\infty$,
\\
which complete the proof. 
\end{proof}
\begin{nemoodar}
We have the following diagram:
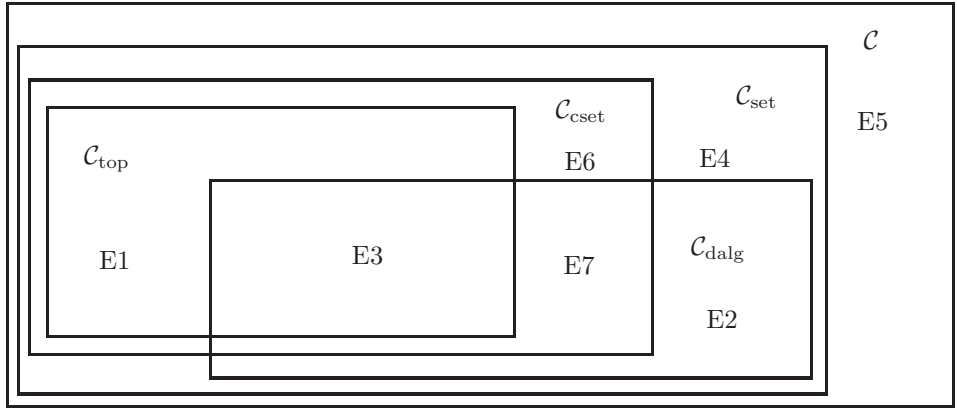
\begin{center}
\unitlength 0.5mm 
\linethickness{0.4pt}
\ifx\plotpoint\undefined\newsavebox{\plotpoint}\fi 
\begin{picture}(254,110.75)(0,0)
\thicklines
\put(3,4.25){\framebox(251,106.5)[]{}}
\put(9,18){\framebox(165,72.5)[]{}}
\put(13.75,22.75){\framebox(123.5,60.5)[]{}}
\put(232.5,101.5){\makebox(0,0)[cc]{$\mathcal C$}}
\put(202,86.25){\makebox(0,0)[cc]{$\mathcal{C}_{\rm set}$}}
\put(155.25,82.5){\makebox(0,0)[cc]{$\mathcal{C}_{\rm cset}$}}
\put(29.25,69.75){\makebox(0,0)[cc]{$\mathcal{C}_{\rm top}$}}
\put(31.25,43){\makebox(0,0)[cc]{E1}}
\put(232.75,79.75){\makebox(0,0)[cc]{E5}}
\put(190.75,70){\makebox(0,0)[cc]{E4}}
\put(98.5,44.25){\makebox(0,0)[cc]{E3}}
\put(155,69.25){\makebox(0,0)[cc]{E6}}
\put(154.75,41.5){\makebox(0,0)[cc]{E7}}
\put(192.75,27.25){\makebox(0,0)[cc]{E2}}
\put(191.75,45.5){\makebox(0,0)[cc]{$\mathcal{C}_{\rm dalg}$}}
\put(6,7.5){\framebox(214.25,92)[]{}}
\put(57,11.75){\framebox(159.25,52)[]{}}
\end{picture}
\end{center}
where by ``Ei'' we mean counterexample $\sigma_{\mu_i}:G^{\Lambda_i}\to G^{\Lambda_i}$
for finite discrete abelian group $G$ with $|G|\geq2$.
\begin{itemize}
\item for $\Lambda_1:={\mathbb Z}$ and $\mu_1:=\lambda_2$ as in Table~\ref{salam99}, we have
	${\rm ent}_{\rm top}(\sigma_{\mu_1})=2\log|G|>0$ and 
	${\rm ent}_{\rm alg}(\sigma_{\mu_1}\restriction_{\mathop{\oplus}\limits_{\Lambda_1}G})=0$,
\item for $\Lambda_2:={\mathbb Z}$ and $\mu_2:=\lambda_1$ as in Table~\ref{salam99}, we have
	${\rm ent}_{\rm alg}(\sigma_{\mu_2}\restriction_{\mathop{\oplus}\limits_{\Lambda_2}G})=
	2\log|G|>0$ and ${\rm ent}_{\rm cset}(\sigma_{\mu_2})=0$,
\item for $\Lambda_3:={\mathbb Z}\times\{0,1\}$ and
	\[\mu_3(n,i)=\left\{\begin{array}{lc} \mu_1(n) & i=0\:, \\ \mu_2(n) & i=1\:, \end{array}\right.\]
	we have ${\rm ent}_{\rm top}(\sigma_{\mu_3})=	{\rm ent}_{\rm alg}(\sigma_{\mu_3}
	\restriction_{\mathop{\oplus}\limits_{\Lambda_3}G})=2\log|G|>0$,
\item for $\Lambda_4:={\mathbb N}$ and $\mu_4=\lambda_3\restriction_{\mathbb N}$ we have
	$	{\rm ent}_{\rm alg}(\sigma_{\mu_4}\restriction_{\mathop{\oplus}\limits_{\Lambda_4}G})={\rm ent}_{\rm cset}(\sigma_{\mu_4})=0$
	and ${\rm ent}_{\rm set}(\sigma_{\mu_4})=\infty$,
\item for $\Lambda_5:={\mathbb Z}$ and $\mu_5(n)=-n$ ($n\in\mathbb{Z}$) we have  
	${\rm ent}_{\rm set}(\sigma_{\mu_5})=0$,
\item for $\Lambda_6:={\mathbb N}$ and $\mu_6=(1,2)(3,4,5)(6,7,8,9)(10,11,12,13,14)\cdots$
	we have 
	$	{\rm ent}_{\rm alg}(\sigma_{\mu_6}\restriction_{\mathop{\oplus}\limits_{\Lambda_6}G})=
	{\rm ent}_{\rm top}(\sigma_{\mu_6})=0$ and ${\rm ent}_{\rm cset}(\sigma_{\mu_6})=\infty$,
\item for $\Lambda_7=(\mathbb{N}\times\{0\})\cup(\mathbb{Z}\times\{1\})$ and
	\[\mu_7(n,i)=\left\{\begin{array}{lc} \mu_6(n) & i=0\:, \\ \mu_2(n) & i=1\:, \end{array}\right.\]
	we have 
	${\rm ent}_{\rm alg}(\sigma_{\mu_7}\restriction_{\mathop{\oplus}\limits_{\Lambda_7}G})=
	2\log|G|>0$, ${\rm ent}_{\rm cset}(\sigma_{\mu_7})=\infty$, 
	${\rm ent}_{\rm set}(\sigma_{\mu_7})=0$.
\end{itemize}
\end{nemoodar}

\noindent {\small 
{\bf Zahra Nili Ahmadabadi},
Islamic Azad University, Science and Research Branch, 
Tehran, Iran
({\it e-mail}: zahra.nili.a@gmail.com)
\\
{\bf Fatemah Ayatollah Zadeh Shirazi},
Faculty of Mathematics, Statistics and Computer Science,
College of Science, University of Tehran ,
Enghelab Ave., Tehran, Iran \linebreak
({\it e-mail}: fatemah@khayam.ut.ac.ir)}

\end{document}